\documentclass[a4paper,12pt]{amsart}
\usepackage{enumerate}
\usepackage{comment}
\usepackage{amsthm,amsmath, amssymb}
\usepackage[top=35truemm, bottom=35truemm, left=30truemm, right=30truemm]{geometry}

\newcommand{\Assouad}{\dim_{\mathrm{A}}}

\newcommand{\Haus}{\dim_{\mathrm{H}}}

\renewcommand{\epsilon}{\varepsilon}
\renewcommand{\limsup}{\varlimsup}

\numberwithin{equation}{section}

\newtheorem{theorem}{Theorem}[section]
\newtheorem{lemma}[theorem]{Lemma}
\newtheorem{corollary}[theorem]{Corollary}

\newtheorem{question}[theorem]{Question}

\theoremstyle{definition}

\newtheorem{notation}[theorem]{Notation}

\title[weak APs and multi-dimensional Szemer'edi's theorem]{New bounds for dimensions of a set\\ uniformly avoiding multi-dimensional\\ arithmetic progressions}
\date{}
\thanks{The author announced the one dimensional cases of the results of this paper in the conferences titled `Research on the Theory of Random Dynamical Systems and Fractal Geometry' in Kyoto University on 31st August, 2019, and `RIMS Workshop 2019, Analytic Number Theory and Related Topics' in Kyoto University on 18th October, 2019. }

\author[K. Saito]{ Kota Saito }
\address{Kota Saito\\
Graduate School of Mathematics\\ Nagoya University\\ Furocho\\ Chikusa-ku\\ Nagoya\\ 464-8602\\ Japan }
\curraddr{}
\email{m17013b@math.nagoya-u.ac.jp}
\thanks{}



\begin{document}
\maketitle

\begin{abstract}
Let $r_k(N)$ be the largest cardinality of a subset of $\{1,\ldots,N\}$ which does not contain any arithmetic progressions (APs) of length $k$. In this paper, we give new upper and lower bounds for fractal dimensions of a set which does not contain $(k,\epsilon)$-APs in terms of $r_k(N)$, where $N$ depends on $\epsilon$. Here we say that a subset of real numbers does not contain $(k,\epsilon)$-APs if we can not find any APs of length $k$ with gap difference $\Delta$ in the $\epsilon \Delta$-neighborhood of the set. More precisely, we show multi-dimensional cases of this result.  As a corollary, we find equivalences between multi-dimensional Szemer\'edi's theorem and bounds for fractal dimensions of a set which does not contain multi-dimensional $(k,\epsilon)$-APs. 
\end{abstract}

\section{introduction}
A real sequence $(a_j)_{j=0}^{k-1}$ is called an arithmetic progression of length $k$ if there exists $\Delta>0$ such that
\[
a_j=a_0+ \Delta j
\]
for all $j=0,1,\ldots, k-1$. We say $\Delta$ is the gap difference of $(a_j)_{j=0}^{k-1}$. It is a big problem to show the existence or non-existence of arithmetic progressions in a given set. Recently, we get great progresses on the problem. For example, Green and Tao proved that the set of prime numbers contains arbitrarily long arithmetic progressions \cite{GreenTao}. 

Let us define an arithmetic patch which is a higher dimensionalized arithmetic progression. Let $\mathbf{v}=\{v_1,\ldots, v_m\}$ be a set of orthogonal unit vectors in $\mathbb{R}^d$ where $1\leq m\leq d$. For every $k\in \mathbb{N}$ and $\Delta>0$, we say that a set $P\subset \mathbb{R}^d$ is an \textit{arithmetic patch (AP)} of size $k$ and scale $\Delta$ with respect to orientation $\mathbf{v}$ if
\[
   P=\left\{t+\Delta\sum_{i=1}^m x_iv_i \ : \ x_i=0,1,\ldots,k-1 \right\}
\]
for some $t\in \mathbb{R}^d$. For every $\epsilon\in [0,1/2)$, we say that $Q \subset \mathbb{R}^d$ is a $(k, \varepsilon, \mathbf{v})${\it -AP} if there exists an arithmetic patch $P$ of size $k$, and scale $\Delta >  0$ with respect to orientation $\mathbf{v}$ such that
\begin{equation}\label{f1.2}
\sup_{x\in P} \inf_{y \in Q} \|x-y\| \leq \epsilon \Delta.
\end{equation}
Note that $(k,0,\mathbf{v})$-APs are arithmetic patches of size $k$ with orientation $\mathbf{v}$. Fraser and Yu gave the original notion of $(k,\epsilon,\mathbf{v})$-APs in \cite{FraserYu}. The term $(k,\epsilon,\mathbf{v})$-APs was firstly seen in \cite{FraserSaitoYu}. The existence of $(k,\epsilon, \mathbf{v})$-APs of a given set $F$ is connected with the Assouad dimension of $F$. Fraser and Yu showed that a subset of $\mathbb{R}^d$ has Assouad dimension $d$ if and only if the set contains $(k, \varepsilon, \mathbf{v})$-APs for every $k\geq 3$, $\epsilon>0$, and basis $\mathbf{v}$. Here the orthogonality of $\mathbf{v}$ does not require in their paper and they consider not only $\mathbb{R}^d$ but also any finitely dimensional Banach spaces. Note that Fraser and Yu say that $F$ aymptotically contains arbitrarirly large arithemetic patches in \cite{FraserYu} instead that $F$ contains $(k, \varepsilon, \mathbf{e})$-APs for every $k\geq 3$, $\epsilon>0$ where $\mathbf{e}$ denotes some fixed basis on a finitely dimensional Banach space. Furthermore, Fraser, the author, and Yu gave the quantitative upper bound of the Assouad dimension of a subset of $\mathbb{R}^d$ which does not contain $(k,\epsilon, \mathbf{v})$-APs as follows:
\begin{theorem}[{\cite[Theorem~5.1]{FraserSaitoYu}}] \label{FraserSYu}
   Fix integers $m$ and $d$ with $1 \leq m\leq d$, and fix $k \geq 2$ and $\epsilon \in (0,1/\sqrt{d})$. Let $F\subseteq \mathbb{R}^d$. If $F$ does not contain $(k,\epsilon,\mathbf{v})$-APs for some a set of orthogonal unit vecters $\mathbf{v}=\{v_1,\ldots, v_m\}$, then we have
   \[
   \dim_{\mathrm{A}} F\leq d+\frac{\log (1-1/k^m)}{\log (k\lceil \sqrt{d}/(2\varepsilon)\rceil)}.
   \]
\end{theorem}

We now define
\begin{align}\label{DefD}
D_\mathrm{A}(k,\epsilon, d,m) &= \sup\{\Assouad F
\colon F \subseteq \mathbb{R}^d, \text{ $F$ does not contain any $(k,\epsilon,\mathbf{v})$-APs}\\\nonumber
&\hspace{50pt} \text{for some a set of orthogonal unit vectors $\{v_1,\ldots,v_m\}$}\},
\end{align}
and
\[
D_\mathrm{A}(k,\epsilon)=D_\mathrm{A}(k,\epsilon,1,1).
\]
We also define $D_\mathrm{H}(k,\epsilon,d,m)$ by replacing $\dim_{\mathrm{A}}$ and $F\subseteq \mathbb{R}^d$ in (\ref{DefD}), to $\dim_{\mathrm{H}}$ and the condition that $F\subset\mathbb{R}^d$ is compact. Here $\dim_{\mathrm{H}} F$ denotes the Hausdorff dimension of $F$. By Theorem~\ref{FraserSYu}, we obtain the upper bound for $D_\mathrm{A}(k,\epsilon,d,m)$. In particular, when $d=m=1$, Fraser, the author and Yu give lower and upper bounds for $D_\mathrm{A}(k,\epsilon)$ and $D_\mathrm{H}(k,\epsilon)$ in \cite{FraserSaitoYu} as follows:
\begin{equation}\label{BoundsForD}
\frac{\log 2}{\log \frac{2k-2-4\epsilon}{k-2-4 \epsilon}} \leq D_\mathrm{H}(k,\epsilon)\leq D_\mathrm{A}(k,\epsilon)\leq  1+\frac{\log (1-1/k)}{\log k \lceil 1/(2\varepsilon)\rceil}
\end{equation}
for every $k\geq 3$ and $\varepsilon \in (0,1/2)$ with $\epsilon <(k-2)/4$.

The goal of this paper is giving new upper and lower bounds for fractal dimensions of a set which does not contain $(k,\epsilon,\mathbf{v})$-APs for some set of orthogonal unit vectors $\{v_1,\ldots, v_m\}$, in terms of the function $r_{k,m}(N)$. Here $r_{k,m}(N)$ denote the largest cardinality of $A\subseteq \{1,\ldots ,N\}^m$ such that $A$ does not contain any arithmetic patches of size $k$ with orientation $\{e_1,\ldots, e_m\}$, where  $e_i$ denotes the vector in $\mathbb{R}^d$ of which $i$-th coordinate is $1$ and others are $0$. Further, we give the equivalent conditions between multi-dimensional Szemer\'edi's theorem given by Furstenberg and Katznelson \cite{FurstenbergKatznelson} and bounds for $D_\mathrm{A}(k,\epsilon,d,d)$.

\begin{notation} We give the following notations:
\begin{itemize}
\item $\mathbb{N}$ denotes the set of all positive integers;
\item for every $F\subseteq \mathbb{R}^d$, $\dim_{\mathrm{L}} F$ denotes the lower dimension of $F$, $\dim_{\mathrm{P}}F$ denotes the packing dimension of $F$, $\dim_{\mathrm{LB}} F$ denotes the lower box dimension of $F$, and $\dim_{\mathrm{UB}} F$ denotes the upper box dimension of $F$;
\item for every $\mathrm{X}\in \{\mathrm{L,H,P,LB,UB}\}$,
define
\begin{align*}
D_\mathrm{X}(k,\epsilon, d,m) &= \sup\{\dim_{\mathrm{X}} F
\colon F \subset \mathbb{R}^d\text{ is compact} , \text{ $F$ does not contain any}\\\nonumber
&\hspace{10pt} \text{$(k,\epsilon,\mathbf{v})$-APs for some set of orthogonal unit vectors $\{v_1,\ldots,v_m\}$}\},
\end{align*}
and
$
D_\mathrm{X}(k,\epsilon)=D_\mathrm{X}(k,\epsilon,1,1).
$
\item for every $x\in \mathbb{R}$, $\lceil x\rceil$ denotes the minimum integer $n$ such that $x\leq n$, and
$\lfloor x \rfloor$ denotes the maximum integer $n$ such that $x\geq n$;
\item for every finite set $A$, $|A|$ denotes the cardinality of $A$.
\end{itemize}

\end{notation}

\section{Result}

\begin{theorem}\label{Multi-main0}
Fix integers $k\geq 2$, $d\geq 1$ and $1\leq m\leq d$, and fix a real number $\epsilon\in (0,1/2)$. If $F\subseteq \mathbb{R}$ does not contain any $(k,\epsilon,\mathbf{v})$-APs for some set of orthogonal unit vectors $\mathbf{v}=\{v_1, \ldots, v_m \}$, then we have
\[
\Assouad F \leq \inf_{N\in \mathbb{N}} \frac{\log(\lceil \sqrt{d}/\epsilon \rceil^d N^{d-m} r_{k,m}(N) ) }{\log (N \lceil \sqrt{d}/\epsilon\rceil)}.
\]
In particular, if we substitute $N=\lceil\sqrt{d}/\epsilon \rceil$, then
\begin{equation}\label{f2-main0}
\Assouad F \leq d+\frac{1}{2} \frac{\log(r_{k,m}(\lceil \sqrt{d}/\epsilon \rceil) /\lceil\sqrt{d}/\epsilon \rceil^m )}{\log\lceil \sqrt{d}/\epsilon \rceil}.
\end{equation}
\end{theorem}
We will prove Theorem~\ref{Multi-main0} in Section~\ref{ProofTheorems}. This gives a better upper bound for $D_\mathrm{A}(k,\epsilon,d,m)$ by the multidimensional Szemer\'edi's theorem if $\epsilon$ is sufficiently small. This will be claimed in Corollary~\ref{explicitupper}.

\begin{theorem}\label{Multi-main1}
Fix integers $k\geq 2$, $d\geq 1$ and $1\leq m\leq d$, where $k\geq 3$ when $d=1$. Fix a real number $0<\epsilon<1/8$. Let $N=\lceil 1/(8\epsilon) \rceil$, $0<\delta\leq 1/24 $ and $A$ be a subset of $\{0,1,\ldots, N-1\}^d$ which does not contain any arithmetic patches of size $k$ with orientation $\{e_1,\ldots,e_m\}$. For all $a\in A$ and $x\in \mathbb{R}$, we define
\begin{gather*}
\phi_a (x)=\frac{\delta }{N-1+\delta}\ x +a.
\end{gather*}
Let $F$ be the attractor of the iterated function system $(\phi_a )_{a\in A}$, that is,
\[
F=\bigcup_{a\in A} \phi_a(F).
\]
Then the following hold:
\begin{itemize}
\item[(i)] the iterated function system $\{\phi_a\colon a\in A\}$ satisfies open set condition;
\item[(ii)] $F$ does not contain any $(k,\epsilon, \{e_1, \ldots, e_m\})$-APs;
\item[(iii)] it follows that
\[
\Haus F =\frac{\log |A|}{\log \left(\frac{N-1}{\delta} +1\right )}.
\]
\end{itemize}
\end{theorem}
We will prove Theorem~\ref{Multi-main1} in Section~\ref{ProofTheorems}. This theorem gives a new lower bound for $D_\mathrm{A}(k,\epsilon,d,m)$. Here a set of contractive functions $\{f_1,\ldots, f_n\}$ from $\mathbb{R}^d$ to $\mathbb{R}^d$ is called an iterated function system on $\mathbb{R}^d$. We say that an iterated function system $\{f_1,\ldots,f_n\}$ on $\mathbb{R}^d$ satisfies open set condition if there exists a bounded open set $V\subset\mathbb{R}^d$ such that
\[
	V\supseteq \bigcup_{i=1}^n f_i(V),
\] 
where the union on the left hand side is pairwise disjoint. The open set condition is useful to calculate the Hausdorff dimension (see \cite{Falconer, Hutchinson} ). We now define
\[
D_{\mathrm{S}}(k,\epsilon,d,m)=\{\dim_{\mathrm{H}}F \colon F\subset\mathbb{R}^d \text{ is compact and satisfies (i) and (ii) in Theorem~\ref{Multi-main1}}  \}.
\]
for every $k\geq 2$, $0<\epsilon<1/2$ and $1\leq m\leq d$. Note that for every bounded set $F\subseteq \mathbb{R}^d$, we have
\begin{equation}\label{FractalDimsInequalities}
\dim_{\mathrm{L}}F \leq \dim_{\mathrm{H}}F \leq \dim_{\mathrm{P}} F\leq \dim_{\mathrm{LB}}F \leq \dim_{\mathrm{UB}} F \leq \dim_{\mathrm{A}}F.
\end{equation}
Further, by Fraser's result \cite{Fraser}, if $F$ satisfies (i) in Theorem~\ref{Multi-main1}, then we have
\begin{equation}\label{FraserFormula}
\dim_{\mathrm{L}}F = \dim_{\mathrm{H}}F = \dim_{\mathrm{P}} F= \dim_{\mathrm{LB}}F = \dim_{\mathrm{UB}} F = \dim_{\mathrm{A}}F.
\end{equation}
Therefore we can replace $\dim_{\mathrm{H}}$ in the definition of $D_{\mathrm{S}}(k,\epsilon, d,m)$ by $\dim_{\mathrm{X}}$ for all $X\in \{\mathrm{L,P,LB,UB,A}\}$. We refer \cite{Falconer,Fraser, Robinson} to the readers who are interested in more details on fractal dimensions.

\begin{corollary}\label{maincoro3}For every $d\geq 1$ and $k\geq 2$ where $k\geq 3$ when $d=1$, and for every $0<\epsilon <1/8$, one has
\begin{align*}
d\left(1-\frac{\log 32}{\log(4/\epsilon)}\right)&+\frac{\log( r_{k,m}(\lceil 1/(8\epsilon)\rceil)/\lceil 1/(8\epsilon)\rceil^{m} ) }{\log(4/\epsilon) } \leq D_\mathrm{S}(k,\epsilon,d,m)\\
&\leq D_\mathrm{A}(k,\epsilon, d, m)
\leq d+\frac{1}{2}\frac{\log(r_{k,m}(\lceil \sqrt{d}/\epsilon\rceil)/\lceil\sqrt{d}/\epsilon \rceil^m) }{\log\lceil\sqrt{d}/\epsilon \rceil}.
\end{align*}
\end{corollary}
We will prove Corollary~\ref{maincoro3} in Section~\ref{ProofCoro} by combining Theorem~\ref{Multi-main0} and Theorem~\ref{Multi-main1}. Recently, in \cite{FraserShmerkinYavicoli}, Fraser, Shmerkin and Yavicoli define
\[
d(k,\epsilon)=\sup\{\dim_{\mathrm{H}}F \colon F\subset\mathbb{R} \text{ is a bounded set which does not contain $(k,\epsilon,\{1\})$-APs} \}.
\]
They prove that 
\begin{align*}
d(k,\epsilon) &=\sup\{\dim_{\mathrm{H}}F \colon F\subset\mathbb{R} \text{ does not contain $(k,\epsilon,\{1\})$-APs} \}\\
	&= \sup\{ \dim_{\mathrm{A}}F \colon F\subset\mathbb{R} \text{ is a bounded set which does not contain $(k,\epsilon,\{1\})$-APs} \}\}.
\end{align*} 
Therefore $D_{\mathrm{H}}(k,\epsilon)=d(k,\epsilon)\leq D_\mathrm{A}(k,\epsilon)$. Further, they give upper and lower bounds for $d(k,\epsilon)$ as follows:
\begin{equation}\label{FShY}
\frac{\log r_{k,1}(\lfloor 1/(10\epsilon)\rfloor)}{\log(10\lfloor 1/(10\epsilon) \rfloor)  } \leq d(k,\epsilon)\leq \frac{1}{2}\left(\frac{\log (r_{k,1}(\lceil1/\epsilon \rceil)+1)}{\log\lceil1/\epsilon\rceil}+\frac{1}{2}\right).  
\end{equation}

These bounds are almost same as the bounds in Corollary~\ref{maincoro3} with $d=m=1$. 

\begin{corollary} Fix any $0<\delta<1$. For every $k\geq 3$, $D_{\mathrm{A}}(k,\epsilon)$ is less than or equal to
\begin{equation}\label{explicitupper}
1-\frac{c_k( \mathrm{L}_3(\epsilon) -\mathrm{L}_2( \epsilon )^{-\delta}) }{(1+\exp(-\mathrm{L}_2( \epsilon )^{1-\delta}))\mathrm{L}_1(\epsilon) +\exp(-\exp(\mathrm{L}_2(\epsilon)(1-\mathrm{L}_2(\epsilon)^{-\delta} )  ) )  }
\end{equation}
for all $0<\epsilon <\epsilon(\delta)$, where we define $\mathrm{L}_1(\epsilon)=\log\lceil 1/\epsilon \rceil,\  \mathrm{L}_n(\epsilon)=\log \mathrm{L}_{n-1}(\epsilon)$ for every $n\geq 2$.
\end{corollary}
We will prove Corollary~\ref{explicitupper} in Section~\ref{ProofCoro}. The first term in the numerator of the complicated fraction in (\ref{explicitupper})
dominates the second, and also the first term of the denominator dominates the second. Hence the right hand side of (\ref{explicitupper}) is near to
\[
1-\frac{c_k \mathrm{L}_3(\epsilon)}{(1+\exp(-\mathrm{L}_2(\epsilon)^{1-\delta}) )\mathrm{L}_1(\epsilon) }
\]
Therefore we obtain better upper bounds if $\delta>0$ is smaller. The upper bound (\ref{explicitupper}) comes from Gowers' upper bound for $r_{k,1}(N)$ \cite{Gowers} as follows:
\begin{equation}\label{GowersBound}
r_{k,1}(N) \leq \frac{N}{(\log\log N)^{c_k}}
\end{equation}
for every $N\geq 3$ and $k\geq 3$, where $c_k=2^{-2^{k+9}}$. In order to simplify, substitute $\delta =1/2$ in (\ref{explicitupper}) and one obtains that
\[
D_{\mathrm{A}}(k,\epsilon)\leq 1-(1+o(1))\frac{c_k \log\log\log\lceil1/\epsilon \rceil }{\log \lceil 1/\epsilon \rceil}
\]
as $\epsilon \rightarrow +0$. This upper bound is better than (\ref{BoundsForD}) if $0<\epsilon<\epsilon(k) $ is sufficiently small.

\begin{corollary}
For every $k\geq 3$ and $0<\epsilon<1/8$, $D_{\mathrm{S}}(k,\epsilon)$ is greater than or equal to
\[
1-\frac{1}{ \log(\frac{4}{\epsilon})} \left(\log (32C) + (\log 2) \left(n2^{(n-1)/2} \sqrt[n]{\log_2 \lceil 1/(8\epsilon)\rceil} +\frac{1}{2n}\log_2\log_2 \lceil 1/(8\epsilon)\rceil \right) \right),
\]
for some absolute constant $C>0$, where $n=\lceil \log_2 k \rceil$.
\end{corollary}

This result immediately comes from Corollary~\ref{maincoro3} with $d=m=1$ and O'Bryant's lower bound for $r_{k,1}(N)$ \cite{OBryant}, which is
\[
r_{k,1}(N) \geq C N\exp \left((\log 2) \left(-n2^{(n-1)/2} \sqrt[n]{\log_2 N} +\frac{1}{2n}\log_2\log_2 N \right) \right)
\]
for all $N\geq 1$ and $k\geq 3$, for some $C>0$. Hence we omit the proof. In order to simplify, for any fixed $k\geq 3$ and for every $0<\epsilon <\epsilon(k)$ we have
\[
D_{\mathrm{S}}(k,\epsilon) \geq    1 - A_k\frac{\sqrt[n]{\log \lceil 1/8\epsilon \rceil }}{\log (1/\epsilon)}
\]
for some constant $A_k>0$ depending on only $k$. This lower bound is better than (\ref{BoundsForD}) if $0<\epsilon<\epsilon(k)$ is sufficiently small.

\begin{corollary}\label{mainineq}
For every $d\geq 1$, there exist positive constants $A_d$ and $B_d$ such that for every integer $k\geq 3$ and real number $0<\epsilon <1/8$, one has
\[
A_d \frac{r_{k,d}(\lceil 1/8\epsilon \rceil)}{\lceil 1/8\epsilon \rceil^d}\leq  \epsilon^{d-D_\mathrm{S}(k,\epsilon,d,d)}\leq \epsilon^{d-D_\mathrm{H}(k,\epsilon,d,d)}\leq \epsilon^{d-D_{\mathrm{A}}(k,\epsilon,d,d)}
\leq B_d \left(\frac{r_{k,d}(\lceil\sqrt{d}/\epsilon \rceil)}{\lceil\sqrt{d}/\epsilon \rceil^d}\right)^{1/2}.
\]
\end{corollary}
We can immediately show this corollary from Corollary~\ref{maincoro3} with $d=m$. Thus we omit the proof. This corollary gives the following equivalences between the multidimensional Szemer\'edi's theorem given by Furstenberg and Katznelson, and bounds for $D(k,\epsilon,d,d)$:

\begin{corollary}\label{mainequiv}
Fix integers $k\geq 3$ and $d\geq 1$. Let $\mathbf{e}$ be the standard basis of $\mathbb{R}^d$. The following are equivalent:
\begin{itemize}
\item[(i)] If $A\subseteq \mathbb{N}^d$ satisfies
\[
\limsup_{N\rightarrow \infty} \frac{|A\cap [1,N]^d|}{N^d}>0,
\]
then $A$ contains $(k,0,\mathbf{e})$-APs;
\item[(ii)] $r_{k,d}(N)/N\rightarrow 0$ as $N\rightarrow \infty$;
\item[(iii)] $\epsilon^{d-D_{\mathrm{S}}(k,\epsilon,d,d)} \rightarrow 0$ as $\epsilon \rightarrow +0$;
\item[(iv)] $\epsilon^{d-D_{\mathrm{H}}(k,\epsilon,d,d)} \rightarrow 0$ as $\epsilon \rightarrow +0$;
\item[(v)] $\epsilon^{d-D_{\mathrm{A}}(k,\epsilon,d,d)} \rightarrow 0$ as $\epsilon \rightarrow +0$.
\end{itemize}
\end{corollary}
We will prove Corollary~\ref{mainequiv} in Section~\ref{ProofCoro}. Furstenberg and Katznelson proved that for any $A\subseteq \mathbb{Z}^d$ satisfying
\begin{equation}\label{posBanachdens}
\lim_{h\rightarrow \infty} \sup\left\{\frac{|A\cap I|}{h^d} \colon I\subset\mathbb{R}^d \text{ is a closed hyper-cube with side length $h$} \right\}>0,
\end{equation}
and for any finite set $F\subset \mathbb{Z}^d$, there exists $a\in \mathbb{Z}^d$ and $\Delta \in \mathbb{Z}$ such that $a+\Delta F \subset A$. This statement is equivalent to (i) in Corollary~\ref{mainequiv} is true for every $d\geq 1$ and $k\geq 3$. Further, $D_{\mathrm{H}}(k,\epsilon, d,d)$ in (iv) in Corollary~\ref{mainequiv} can be replaced other fractal dimensions. Therefore for each $\mathrm{X}\in \{\mathrm{S,L,H,P,LB,UB,A}\}$, the multidimensional Szemer\'edi's theorem is equivalent to
\[
\lim_{\epsilon\rightarrow+0}\epsilon^{d-D_X(k,\epsilon,d,d)}=0
\]
for every $k\geq 3$ and $d\geq 1$.

\section{Proof of Corollaries}\label{ProofCoro}

Let $d$ and $m$ be integers with $1\leq m\leq d$, and let $k\geq 2$ be integer. For every integer $N\geq 1$, we define $r_{k,d,m}(N)$ as the largest cardinality of $A\subseteq \{1,\cdots, N \}^d$ such that $A$ does not contain any arithmetic patches of size $k$ with orientation $\{e_1,\ldots, e_m \}$.

\begin{lemma}\label{rkm} For every $1\leq m\leq d$, $k\geq 2$, and $N\geq 1$, we have
\[
r_{k,d,m}(N)=N^{d-m}r_{k,m}(N).
\]
\end{lemma}
\begin{proof}
Let $A\subseteq \{1,\ldots, N\}^d$ which does not contain $(k,0,\{e_1,\ldots, e_m\})$-APs and $|A|=r_{k,m}(N)$. Define
\[
B=\{(x,y)\in\{1,\ldots,N\}^m \times \{1,\ldots,N\}^{d-m} \colon x\in A, y\in \{1,\ldots,N\}^{d-m}  \}
\]
Then $B$ is a subset of $\{1,\ldots, N\}^d$ and does not contain $(k,0,\mathbf{v})$-APs. Therefore one has
\[
r_{k,d,m}(N)\leq |B| =N^{d-m} r_{k,m}(N).
\]
On the other hand, take any $A\subseteq \{1,\cdots, N\}^d$ which does not contain any $(k,\epsilon, \mathbf{e})$-APs. For every $1\leq j_{m+1},\ldots, j_{d}\leq N $, define
\[
B_{j_{m+1},\cdots,j_{d}}=\{x\in\{1,\ldots, N\}^m \colon (x,j_{m+1},\ldots, j_d)\in A \}.
\]
Then each $B_{j_{m+1},\ldots, j_d}$ does not contain any $(k,\epsilon, \mathbf{e})$-APs. Hence we obtain
\[
N^{d-m}r_{k,m}(N)\leq \sum_{j_{m+1},\ldots, j_{d}} |B_{j_{m+1},\ldots, j_d}| = |A|=r_{k,d,m}(N).
\]
\end{proof}

\begin{proof}[Proof of Corollary~\ref{maincoro3}]
By (\ref{f2-main0}) in Theorem~\ref{Multi-main0}, one has
\[
D_\mathrm{A}(k,\epsilon, d, m)
\leq d+\frac{1}{2}\frac{\log(r_{k,m}(\lceil \sqrt{d}/\epsilon\rceil)/\lceil\sqrt{d}/\epsilon \rceil^m) }{\log\lceil\sqrt{d}/\epsilon \rceil}.
\]
for all $k\geq 2$, $0<\epsilon< 1/8$ and $1\leq m\leq d$. We next find a lower bound for $D_{\mathrm{S}}(k,\epsilon,d,m)$. Let $N=\lceil 1/(8\epsilon) \rceil $. Take $B\subseteq \{1,\ldots, N\}^d$ which does not contain $(k,0,\{e_1,\ldots,e_m\})$-APs and $|B|=r_{k,d,m}(N)$. By Lemma~\ref{rkm} and Theorem~\ref{Multi-main1} with $A=B$ and $\delta=1/24$, one has
\begin{align*}
D_\mathrm{S}(k,\epsilon, d, m)&\geq \frac{\log r_{k,d,m}(\lceil 1/(8\epsilon) \rceil)}{\log (24(\lceil 1/(8\epsilon)\rceil-1)+1) } \geq \frac{\log (r_{k,m}(\lceil 1/(8\epsilon)\rceil)\lceil 1/(8\epsilon)\rceil^{d-m}) }{\log (4/\epsilon)}\\
&\geq d\left(1-\frac{\log 32}{\log(4/\epsilon)}\right)+\frac{\log( r_{k,m}(\lceil 1/(8\epsilon)\rceil)/\lceil 1/(8\epsilon)\rceil^{m} ) }{\log(4/\epsilon) } .
\end{align*}

\end{proof}

\begin{proof}[Proof of Corollary~\ref{explicitupper}]
Fix $0<\delta<1$, $k\geq 3$ and let $0<\epsilon \ll_\delta 1$ be a sufficiently small real number. Choose $N= \lceil \lceil 1/\epsilon \rceil^r \rceil $ where $r(\epsilon)=\exp(-(\log\log\lceil 1/\epsilon\rceil)^{1-\delta} )$ and $\delta=1/2$. By Theorem~\ref{Multi-main0} with  $d=m=1$, we have
\begin{align*}
D_{\mathrm{A}}(k,\epsilon,1)&\leq \frac{\log (r_k(N)\lceil 1/\epsilon\rceil ) }{\log(N\lceil 1/\epsilon \rceil ) }\leq 1-\frac{c_k \log\log (r\log \lceil 1/\epsilon \rceil  ) }{(r+1)\log \lceil1/\epsilon \rceil +\log(1+\lceil 1/\epsilon \rceil^{-r})} \\
& \leq 1- \frac{c_k \log (\mathrm{L}_2(\epsilon) - \mathrm{L}_2(\epsilon)^{1-\delta}  )     }{(1+\exp(-\mathrm{L}_2(\epsilon)^{1-\delta}))\mathrm{L}_1(\epsilon) +\lceil 1/\epsilon \rceil ^{-r}},
\end{align*}

which implies that $D_{\mathrm{A}}(k,\epsilon)$ is less than or equal to
\begin{equation}
1-\frac{c_k( \mathrm{L}_3(\epsilon) -\mathrm{L}_2( \epsilon )^{-\delta}) }{(1+\exp(-\mathrm{L}_2( \epsilon )^{1-\delta}))\mathrm{L}_1(\epsilon) +\exp(-\exp(\mathrm{L}_2(\epsilon)(1-\mathrm{L}_2(\epsilon)^{-\delta} )  ) )  } .
\end{equation}
\end{proof}

\begin{proof}[Proof of Corollary~\ref{mainequiv}]
By Corollary~\ref{mainineq}, (ii)-(v) are equivalent. Thus it suffices to show that (i) and (ii) are equivalent. The following lemma implies this equivalence:

\begin{lemma}\label{Lemma1}
Fix $k\geq 2$ and $d\geq 1$, and let $\mathbf{e}$ be the standard basis on $\mathbb{R}^d$. The following are equivalent:
\begin{itemize}
\item[(i)] $r_{k,d}(N)/N^d \rightarrow 0$ as $N\rightarrow 0$;
\item[(ii)] Any $A\subseteq \mathbb{Z}^d$ with (\ref{posBanachdens})
contains $(k,0,\mathbf{e})$-APs;

\item[(iii)]
If $A\subseteq \mathbb{N}^d$ satisfies
\begin{equation}\label{ap-f6}
\limsup_{N\rightarrow \infty} \frac{|A\cap [1,N]^d|}{N^d}>0,
\end{equation}
then $A$ contains $(k,0,\mathbf{e})$-APs.
\end{itemize}
\end{lemma}

We prove this lemma in Appendix.
\end{proof}

\section{Proof of main Theorems}\label{ProofTheorems}
For every $x\in \mathbb{R}^d$ and $R>0$, $B(x,R)$ denotes the closed ball with radius $R$ centered at $x\in \mathbb{R}^d$. For every bounded set $E\subset \mathbb{R}^d$ and $r>0$, $N(E,r)$ denotes the smallest cardinality of a family of sets whose diameters are less than or equal to $r$.
The \textit{Assouad dimension} of $F\subseteq \mathbb{R}^d$ is defined by
\begin{align*}
   \Assouad F = \inf \Big\{\;\sigma\geq 0 &\colon \exists C>0\ \forall r>0\ \forall R> r \ \forall x\in F\\
   &\, \hspace{2cm}\  N\Bigl(B(\alpha, R)\cap F,r\Bigr)\leq  C R^\sigma \Big\}.
\end{align*}
By this definition, we obtain that for every $F\subseteq \mathbb{R}^d$
\begin{align}\label{AssouadDef2}
   \Assouad F = \inf \Big\{\;\sigma\geq 0 &\colon \exists C>0 \ \exists\lambda\geq 1\ \forall r>0\ \forall R> \lambda r \ \forall x\in F\\ \nonumber
   &\, \hspace{2cm}\  N\Bigl(B(\alpha, R)\cap F,r\Bigr)\leq  C R^\sigma \Big\}.
\end{align}

\begin{proof}[Proof of Theorem~\ref{Multi-main0}]

Choose any set $F$ which does not contain $(k,\epsilon,\{v_1,\ldots, v_m\})$-APs. By rotating, we may assume that
$v_1=e_1, \ldots, v_m=e_m$. Suppose that $\sqrt{d}/\epsilon $ is an integer. Fix any small real number $\alpha$ and large parameter $\lambda=\lambda(\alpha)$, and fix any $r,R$ with $R/r>\lambda$. Fix a ball $B$ of $\mathbb{R}^d$ with radius $R$ and centered at a point in $F$. Choose a hyper-cube $C\supseteq B$ with side length $2R$. Write
\[
   C=\prod_{i=1}^d [a_i,a_i+2R].
\]
Fix any positive integer $N$. For every $i=1,2,\ldots,d$ and $j=0,1,\ldots, \sqrt{d}N/\epsilon-1$, we define
\begin{gather*}
A^{(i)}_{j}=[a_i+2j R\epsilon/(N\sqrt{d}), a_i +2(j+1) R\epsilon/(N\sqrt{d})].
\end{gather*}
Let $c^{(i)}_j$ be the middle point of $A^{(i)}_j$ for all $i$ and $j$. Let us find a family of hyper-cubes with side length $2R\epsilon/(N\sqrt{d})$ which covers $F\cap B$ and whose cardinality is less than or equal to
\begin{equation}\label{nf1}
\left(\frac{\sqrt{d}}{\epsilon}\right)^d r_{k,d,m}(N).
\end{equation}
Here define
\begin{gather*}
I(j_1,\ldots,j_d)=\{(j_1+\sqrt{d} n_1/\epsilon, \ldots, j_d+\sqrt{d}n_d/\epsilon)\in\mathbb{Z}^d\colon 0\leq n_1,\ldots, n_d\leq N-1 \}
\end{gather*}
for every $1\leq j_1, \ldots, j_d \leq \sqrt{d}/\epsilon$. Note that
\[
\bigcup_{1\leq j_1,\ldots, j_d\leq \sqrt{d}N/\epsilon} I(j_1,\ldots, j_d)=\{1,2,\ldots, \sqrt{d}N/\epsilon\}^d,
\]
which is a disjoint union. Fix any $1\leq j_1, \ldots, j_d \leq \sqrt{d}/\epsilon$. Let
\begin{gather*}
I=I(j_1,\ldots, j_d),\
P=\{(c^{(1)}_{x_1}, \ldots, c^{(d)}_{x_d})\in \mathbb{R}^d\colon (x_1,\ldots, x_d)\in I \},\\
\mathcal{S}=\left\{\prod_{i=1}^d A^{(i)}_{x_i}\colon (x_1,\ldots, x_d)\in I \right\}.
\end{gather*}
Assume that the number of $A\in \mathcal{S} $ such that  $F\cap A\neq \emptyset$ is at least $r_{k,d,m} (N)+1$. Then we can find an arithmetic patch $Q\subseteq P$ of size $k$ and scale $\Delta\geq 2R/N$ with orientation $\{e_1,\ldots, e_m\}$ satisfying that for all $x\in Q$, there exists $y=y(x)\in F$ such that
\[
\|x-y\| \leq 2R\epsilon/N \leq \epsilon \Delta.
\]
Thus $\{y(x)\colon x\in Q\}$ is a $(k,\epsilon,\mathbf{e})$-AP. This is a contradiction. Hence the number of $A\in \mathcal{S}$ such that $F\cap A\neq \emptyset$ is less than or equal to $r_{k,d,m} (N)$ for each fixed $1\leq j_1, \ldots, j_d \leq \sqrt{d}/\epsilon$. Therefore one has (\ref{nf1}). We iterate this argument $t$-times for each smaller hyper-cubes which intersect $F$. Here $t$ is a positive integer which is determined later. Then the number of hyper-cubes with side length $2R(\epsilon/(N\sqrt{d}))^t $ which covers $F$ is less than or equal to $(\sqrt{d}/\epsilon)^{dt} r_{k,d,m}(N)^t$. Let
\[
t=\left\lceil \frac{\log (2R\sqrt{d}/r)}{\log (N\sqrt{d}/\epsilon)} \right\rceil.
\]
Then one has
\[
2R\sqrt{d}\left(\frac{\epsilon}{N\sqrt{d}} \right)^t
\leq  2R\sqrt{d} \left(\frac{\epsilon}{N\sqrt{d}} \right)^{\frac{\log (r/(2R\sqrt{d}))}{\log (\epsilon/(N\sqrt{d}))}}=r.
\]
Therefore we obtain that
\begin{align*}
N(F\cap B, r) &\leq N(F\cap C, r)\leq  \left(
\left(\frac{\sqrt{d}}{\epsilon}\right)^{d} r_{k,d,m}(N)
\right)^t\\
&\leq \left(\frac{2R\sqrt{d}}{r} \right)^{(1+\alpha)\frac{\log((\sqrt{d}/\epsilon)^d r_{k,d,m}(N) ) }{\log (N\sqrt{d}/\epsilon )}}.
\end{align*}
Hence by (\ref{AssouadDef2}), we conclude that
\[
\Assouad F \leq (1+\alpha)\frac{\log((\sqrt{d}/\epsilon)^d r_{k,d,m}(N) ) }{\log (N\sqrt{d}/\epsilon )},
\]
which implies that
\[
\Assouad F \leq \frac{\log((\sqrt{d}/\epsilon)^d r_{k,d,m}(N) ) }{\log (N\sqrt{d}/\epsilon )}
\]
as $\alpha \rightarrow +0$.

If $\sqrt{d}/\epsilon$ is not an integer, then let
\[
\epsilon'= \frac{\sqrt{d}}{\lceil \sqrt{d}/\epsilon \rceil}.
\]
It is seen that $\epsilon' \leq \epsilon$. Therefore $F$ does not contain $(k,\epsilon', \{v_1,\ldots, v_m\})$-APs for some a set of orthogonal unit vectors $\{v_1,\ldots, v_m\}$. Hence one has
\[
\Assouad F \leq \frac{\log(\lceil \sqrt{d}/\epsilon \rceil^d N^{d-m} r_{k,m}(N) ) }{\log (N \lceil \sqrt{d}/\epsilon\rceil)}
\]
by Lemma~\ref{rkm} and the assumption that $\sqrt{d}/\epsilon'$ is an integer.
\end{proof}

\begin{proof}[Proof of Theorem~\ref{Multi-main1}]
Let $A\subseteq \{0,\ldots, N-1\}^d$ be a set which does not contain any arithmetic patches of size $k$ with orientation $\{e_1,\ldots, e_m\}$. Define
\begin{gather*}
   \phi_a(x)=\frac{\delta}{N-1+\delta}x +a\quad (a\in A,\ x\in \mathbb{R}^d),\\
I_0=[0,N-1+\delta]^d,\
I_{n+1}=\bigcup_{a\in A} \phi_a (I_n)\ (n\geq 0),\
F=\bigcap_{n=1}^\infty I_n.
\end{gather*}
Then it follows that $I_n\supseteq I_{n+1} $ for every $n\geq 0$. In fact, for all $(x_1,\ldots, x_d)\in I_0$ and $(a_1,\ldots, a_d)\in A$, one has
\[
0\leq \frac{\delta}{N-1+\delta}x_i +a_i\leq N-1+\delta,
\]
which means that $I_1\subseteq I_0$. If $I_{n+1}\subseteq I_n$ holds for some $n\geq 0$, then we have
\[
I_{n+2}\subseteq \bigcup_{a\in A} \phi_a(I_{n})=I_{n+1}.
\]
The set $F$ is the attractor of $\{\phi_a\colon a\in A\}$ since if $F'$ denotes the attractor of $\{\phi_a\colon a\in A \}$, then by the triangle inequality and the monotonicity of $(I_{n})_{n\geq 0}$, one has
\[
d_{\mathrm{H}}(F,F')\leq d_{\mathrm{H}} (I_n,F)+d_{\mathrm{H}} (I_n,F')\leq d_{\mathrm{H}}(I_n, F)+\left(\frac{\delta}{N-1}\right)^n d_{\mathrm{H}}(I_0, F')\rightarrow +0
\]
as $n\rightarrow \infty$. Here $d_\mathrm{H}(A,B)$ denotes the Hausdorff metric between compact sets $A$ and $B$ of $\mathbb{R}^d$. Therefore $F'=F$. The iterated function system $\{\phi_a\colon a\in A\}$ satisfies open set condition since one has
\[
(0,N-1+\delta)^d \supseteq \bigcup_{a\in A} \phi_a((0,N-1+\delta)^d),
\]
and the union on the right hand side is disjoint.
This yields that
\[
\Haus F =\frac{\log |A|}{\log (\frac{N-1}{\delta}+1)}
\]
by Hutchinson's theorem (alternatively see \cite[Theorem~9.3]{Falconer}). The remaining part is to show that $F$ does not contain $(k,\epsilon,\{e_1,\ldots, e_m\})$-APs. Let $\mathbf{e}=\{e_1,\ldots, e_m\}$. Assume that $F$ contains a $(k,\epsilon, \{e_1,\ldots, e_m\})$-APs. Let $Q$ be such a $(k,\epsilon,\{e_1,\ldots, e_m\})$-APs. It suffices to show that
\begin{equation}\label{reduction1}
Q\subseteq \phi_{a_0}(I_0)
\end{equation}
for some $a_0\in A$. If (\ref{reduction1}) is true, then $\phi_{a_0}^{-1}(Q)\subseteq I_0$ and $\phi^{-1}_{a_0}(Q)$ is also a $(k,\epsilon,\mathbf{e})$-AP. Thus there exists $a_1\in A$ such that
\[
\phi_{a_0}^{-1}(Q)\subseteq \phi_{a_1}(I_0)
\]
which implies that $\phi_{a_1}^{-1}\circ \phi_{a_0}^{-1}(Q)\subseteq I_0$. We iterate this argument $t$-times for any positive integer $t$. Then there exists $a_0,\ldots, a_t\in A$ such that
\begin{equation*}
Q\subseteq \phi_{a_0} \circ \cdots \circ \phi_{a_t} (I_0).
\end{equation*}
The diameter of the right hand side goes to $0$ as $t\rightarrow \infty$. This is a contradiction. Let us show that (\ref{reduction1}). By the definition of $F$, $Q\subseteq I_1$. Hence for all $x\in Q$ there exists $a(x)\in A$ such that
\begin{equation}\label{appro1}
\|a(x)-x\|_\infty \leq \delta.
\end{equation}
Here for every $x=(x_1,\ldots, x_d)\in \mathbb{R}^d$, $\|x\|_2$ denotes the Euclidean norm and $\|x\|_\infty=\max\{|x_i| \colon 1\leq i\leq d\}$.
By definition, there exists $\Delta>0$ and an AP of size $k$ and scale $\Delta$ with orientation $\mathbf{e}$ such that
\begin{equation}\label{appro2}
\inf_{y\in P}\|x-y\|_{\infty}\leq \inf_{y\in P} \|x-y\|_2 \leq \epsilon \Delta
\end{equation}
for all $x\in Q$. Let $y(x)$ be the point $y\in P$ which satisfies (\ref{appro2}) for every $x\in Q$. Here recall that $Q\subseteq I_1 =\bigcup_{a\in A} \phi_a (I_0)$. Fix any $x\in Q$ and choose $x' \in Q$ such that
\[
\|y(x)-y(x')\|_\infty=\Delta
\]
where $x$ and $x'$ are distinct. Then by (\ref{appro1}) and (\ref{appro2}), one has
\begin{align*}
|\|a(x')-a(x)\|_\infty-\Delta| &\leq \|a(x')-y(x')-(a(x)-y(x))\|_\infty\\
&\leq \|a(x')-x'\|_\infty+\|x'-y(x')\|_\infty +\|a(x)-x\|_\infty+\|x-y(x)\|_\infty\\
&\leq 2(\delta +\epsilon \Delta).
\end{align*}
Since $Q\subseteq I_0$, one has
\[
N-1+\delta\geq (k-1)\Delta -2\epsilon \Delta,
\]
which implies that
\[
\Delta \leq \frac{N-1+\delta}{k-1+2\epsilon}.
\]
Since $k\geq 2$, $0<\epsilon <1/8$ and $0<\delta\leq 1/24$,
\begin{align}\label{Delta1}
|\|a(x')-a(x)\|_\infty-\Delta|
&\leq 2\left(\delta+\frac{\lceil 1/(8\epsilon)\rceil-1+\delta}{k-1-2\epsilon}\epsilon\right)\\\nonumber
&\leq 2\left( \delta+\frac{1/(8\epsilon)+\delta}{1-1/4}\epsilon \right)<1/2
\end{align}
Since $a(x'),a(x)\in \mathbb{Z}^d$, $\|a(x')-a(x)\|_\infty \in \mathbb{Z}$. Therefore by (\ref{Delta1}), $\|a(x)-a(x')\|$ is a constant which does not depend on $x$ or $x'$, which implies that $\{a(x)\colon x\in Q\}$ is an AP of size $k$ with orientation $\mathbf{e}$. This is a contradiction. Hence at least two points $x,x'\in Q$ belong to $\phi_a (I_0)$ for some $a\in A$. This yields that
\[
\Delta \leq \delta +2\epsilon \Delta,
\]
which implies that
\begin{equation}\label{evaluateDelta}
\Delta \leq \frac{\delta}{1-2\epsilon}.
\end{equation}
Take $x''\in Q\setminus \{x,x'\}$ such that
\[
\mathrm{dist}(\{x,x'\}, Q\setminus \{x,x'\} )=\mathrm{dist}(\{x,x'\}, \{x''\}),
\]
where $\mathrm{dist} (A,B) =\inf\{\|x-y\|\colon x\in A, y\in B \}$ for every $A,B\subseteq \mathbb{R}^d$. Thus by (\ref{evaluateDelta}), one has
\[
\min\{\|x-x''\|, \|x'-x''\|\}
\leq (1+2\epsilon)\Delta \leq 3\delta <1-\delta.
\]
Therefore $x''$ does not reach to other islands $\phi_{a'}(I_0)$ ($a'\in A\setminus \{a\}$), which means that $x''$ must belong to $\phi_a(I_0)$. By replacing $\{x,x'\}$ to $\{x,x',x''\}$, we can iterate the same argument until the number of $x\in Q$ such that $x\in \phi_a (I_0)$ reaches $|Q|$. Therefore we get $Q\subseteq \phi_a(I_0)$.
\end{proof}

\section{Discrete Analogue}

For every $F\subseteq \mathbb{N}$, define
\[
\mathrm{Dim}_\zeta F =\limsup_{N\rightarrow \infty}\frac{\log |F\cap [1,N]|}{\log N} =\inf\left\{\sigma \geq 0 \colon \sum_{n\in F}n^{-\sigma} <\infty   \right\},
\]
which is introduced by Doty, Gu, Lutz, Mayordomo, and Moser in \cite{DGLMM}, and generalized to a metric space by the author in \cite{Saito}. We can see that
\begin{equation}\label{zetaAssouad}
\mathrm{Dim}_\zeta F \leq \dim_{\mathrm{A}} F
\end{equation}
for all $F\subseteq \mathbb{N}$ by the definition of the Assouad dimension. The author showed the inequality (\ref{zetaAssouad}) more generally in \cite{Saito}. Define
\[
D_\zeta (k,\epsilon)=\sup\{\mathrm{Dim}_\zeta F \colon F\subseteq \mathbb{N} \text{ does not contain any $(k,\epsilon, \{1\})$-APs} \}.
\]
By Theorem~\ref{Multi-main0} with $d=m=1$, one has
\[
D_\zeta (k,\epsilon) \leq D_{\mathrm{A}} (k,\epsilon)\leq \frac{1}{2}\frac{\log(r_k(\lceil 1/\varepsilon \rceil)\lceil 1/\varepsilon \rceil)}{\log(\lceil 1/\varepsilon \rceil)}.
\]
\begin{theorem}\label{discreteMain}
Fix $k\geq 3$ and $\epsilon\in (0,1/16)$. Let $N=\lceil 1/(8\epsilon)\rceil$, $\eta$ be an integer with $\eta\geq 6$ and $A$ be a subset of $\{0,1,\ldots, N-1\}$ with $0\in A$ which does not contain any arithmetic progressions of length $k$.
Define
\begin{gather*}
\psi_a(x) =( \eta +1 )(N-1)x +a\quad  (a\in A, \ x\in \mathbb{Z})\\
B_0=\{0\},\ B_n=\bigcup_{a\in A} \psi_a(B_{n-1})\ (n\geq 1),\ F=\bigcup_{n=0}^\infty B_n.
\end{gather*}
Then the following hold:
\begin{itemize}
\item[(i)]it follows that
\[
F\subseteq \mathbb{N}\cup \{0\} \text{ and }  F=\bigcup_{a\in A} \psi_a(F);
\]
\item[(ii)] $F$ does not contain any $(k,\epsilon,\{1\})$-APs;
\item[(iii)]it follows that
\[
\limsup_{N\rightarrow +\infty} \frac{\log |F\cap [1,N]|}{\log N} \geq \frac{\log\, |A|}{\log ((1+\eta)(N-1))}.
\]
\end{itemize}
\end{theorem}
We can find a set $A\subseteq \{0,1,\ldots, N-1\}$ with $0\in A$ and $|A|=r_{k,1}(N)$ since if $A\subseteq \{0,1,\ldots, N-1\}$ does not contain any arithmetic progressions of length $k$, then $(-\min A)+ A $ is a subset of $\{0,1,\ldots, N-1\}$ with $0\in A$ which does not contain arithmetic progression of length $k$. Therefore we have
\begin{equation}
\frac{\log r_k(\lceil 1/(8\epsilon)\rceil)}{\log  (1/\epsilon)} \leq D_\zeta (k,\epsilon) \leq D_\mathrm{A}(k,\epsilon) \leq \inf_{N\in \mathbb{N}} \frac{\log(r_k(N)\lceil 1/\epsilon \rceil)}{\log(N \lceil 1/\epsilon \rceil)}
\end{equation}
for every $k\geq 3$ and $0<\epsilon<1/16$. Hence we get the following discrete analogue of Corollary~\ref{mainequiv}.

\begin{corollary} Fix $k\geq 3$. Any $A\subseteq \mathbb{N}$ with positive upper density contains arithmetic progressions of length $k$ if and only if
\[
\lim_{\epsilon \rightarrow +0} \epsilon^{1-D_\zeta (k,\epsilon)}=0.
\]

\end{corollary}

\begin{proof}[Proof of Theorem~\ref{discreteMain}]
We can easily show (i) in Theorem~\ref{discreteMain} by definition and the fact that $B_0\subseteq B_1$. Let us show that (iii). Let $N'=N-1$ and $\xi=(\eta+1)N' $. For all $n\geq 1$, it follows that
\begin{align*}
\mathrm{diam} (B_{n}) &\leq \xi\, \mathrm{diam}(B_{n-1}) +N'\leq \xi^2 \mathrm{diam} B_{n-2} + \xi N'+N'\\
&\leq \cdots \leq (\xi^{n-1}+\xi^{n-2}+\cdots+1)N'\leq (6/5)\xi^{n-1} N'.
\end{align*}
Hence one has
\[
|F\cap[0,(6/5)\xi ^{n-1}N']| \geq |B_n \cap [0,(6/5)\xi^{n-1}N']| \geq |A|^n,
\]
since the union
\[
B_n=\bigcup_{a\in A} \psi_a(B_{n-1})
\]
is disjoint for every $n\geq 1$. Therefore we obtain
\[
\limsup_{N\rightarrow \infty} \frac{\log |F\cap [0,N ]}{\log N}\geq  \frac{\log |A|}{ \log((1+\eta)(N-1)) }
\]
as $N\rightarrow \infty$. Let us next show (ii). It follows that
\begin{equation}\label{dMf1}
B_n=\bigcup_{a\in A} (B_{n-1} + \xi^{n-1}a)
\end{equation}
for all $n\geq 1$. This is clear when $n=1$. Assume that (\ref{dMf1}) holds for some $n\geq 1$. Then we have
\begin{align*}\label{dMf2}
B_{n+1} &=\bigcup_{a\in A} \psi_{a} (B_{n})
=\bigcup_{a\in A} \left(\left(\bigcup_{a'\in A} \xi B_{n-1} +\xi^n a' \right)+ a\right)\\
&= \bigcup_{a'\in A}
\left(\left(\bigcup_{a\in A}\eta_N B_{n-1}+a  \right)+\xi^na'\right)
= \bigcup_{a\in A} (B_{n} + \xi^{n}a).
\end{align*}
Assume that $F$ contains a $(k,\epsilon,\{1\})$-AP. Let $P$ be such a $(k,\epsilon,\{1\})$-AP.  Then we can find $n\geq 0$ such that $P\subseteq B_n$. By a similar discussion of the proof of Theorem~\ref{Multi-main1}, $P\subseteq B_{n-1}+\xi^n a$ for some $a\in A$. Hence $B_{n-1}$ contains a $(k,\epsilon,\{1\})$-AP. By iterating this discussion, we conclude that $B_0$ contains a $(k,\epsilon,\{1\})$-AP. This is a contradiction.
\end{proof}

\section{Further discussion}

\begin{question}\label{ErdosTuran}
Is it ture that 
\[
D_\mathrm{S}(k,\epsilon) \leq 1-\frac{\log(\log\lceil 1/(8\epsilon)\rceil (\log\log\lceil 1/(8\epsilon)\rceil)^2)  }{\log (1/\epsilon)}
\]
for every $k\geq 3$ and $0<\epsilon<\epsilon(k)$?
\end{question}
Erd\H{o}s-Tur\'an conjecture states that a subset of positive integers whose sum of reciplocals diverges would contain arbitrarily long arithmetic progressions. This conjecture is still open even if the length of arithmetic progressions is equal to $3$. By partial summation, if for every $k\geq 3$, there exists $C_k>0$ such that for all $N\geq 2$
\[
r_k(N)\leq C_k \frac{N}{\log N (\log \log N)^2},
\]  
then Erd\H{o}s-Tur\'an conjecture would be ture (see \cite{Gowers2}). Therefore by combining this implication and Corollary~\ref{mainineq} with $d=1$, the affirmative answer to Question~
\ref{ErdosTuran} implies the Erd\H{o}s-Tur\'an conjecture.
\begin{question}\label{Szemeredi}
Can we prove that
\[
\lim_{\epsilon\rightarrow+0} \epsilon^{1-D_\mathrm{X}(k,\epsilon)}=0
\]
for all $k\geq 3$ for some $\mathrm{X}\in \{\zeta, \mathrm{L, H,P,LB,UB,A}\}$, by using fractal geometry?
\end{question}
By Corollary~\ref{mainequiv}, the affirmative answer to Question~\ref{Szemeredi} gives another proof of Szemer\'edi's theorem \cite{Szemeredi}.

\appendix
\section{Proof of Lemma~\ref{Lemma1} }

\begin{proof}
Let us show that (i) implies (ii). Let $\delta$ be the left hand side of (\ref{posBanachdens}). Then there exist infinitely many hyper-cubes $I_1,I_2,\ldots$ such that
$h_1<h_2<\cdots \rightarrow \infty $, and
\[
\frac{|A\cap I_n|}{h_n^d}>\frac{\delta}{2}
\]
for all $n\in \mathbb{N}$, where $h_n$ denotes the side length of $I_n$. For sufficiently large $n$, we have
\begin{equation}\label{ap-f2}
|A\cap I_n|\geq \frac{\delta}{4}\lceil h_n \rceil ^d > \frac{r_{k,d}(\lceil h_n \rceil)}{\lceil h_n \rceil ^d} \lceil h_n \rceil ^d = r(\lceil h_n \rceil).
\end{equation}
Here there exists $t\in \mathbb{Z}^d$ such that
\begin{equation}\label{ap-f5}
t+A\cap I_n \subseteq \{1,\ldots, \lceil h_n \rceil\}^d.
\end{equation}
By combing (\ref{ap-f2}) and (\ref{ap-f5}), $t+A\cap I_n$ contains $(k,0,\mathbf{e})$-APs, which implies that $A$ contains $(k,0,\mathbf{e})$-APs.

It is clear that (ii) implies (iii). Therefore let us prove that (iii) implies (i). This is clear when $d=1$ and $k=2$. Thus we discuss the cases when $d=1$ and $k\geq 3$, or $d\geq 2$ and  $k\geq 2$. Fix $d\geq 1$. Assume that $r_{k,d}(N)/N^d$ does not go to $0$ as $N\rightarrow \infty$. Let us construct a subset of integers which satisfies (\ref{ap-f6}) and does not contain any $(k,0,\mathbf{e})$-APs.
We find a positive real number $\delta$ and an infinite sequence $N_1<N_2<\cdots$ of integers such that
\[
r_{k,d}(N_j)> \delta N_j
\]
for every $j\in \mathbb{N}$. Then for every $j\in \mathbb{N}$, choose $A_j\subseteq \{1,2,\ldots, N_j\}^d$ which does not contain any $(k,0,\mathbf{e})$-APs and $|A_j|=r_{k,d}(N)$. Let $t_1$ be the origin of $\mathbb{R}^d$, $B_1=t_1+A_1$ and $M_1=N_1$. It is clear that $B_1\subset [1, M_1]$. Assume that we have an increasing sequence of sets $B_1\subset B_2\subset \cdots \subset B_n$ and integers $M_1<M_2<\cdots< M_n$ such that $B_n\subset[1,M_n]^d$. Then take $N_{j_{n+1}}$ with $N_{j_{n+1}}>M_n$, and let
\begin{gather*}
t_{n+1}=(M_n+2N_{j_{n+1}})e_1,\quad  B_{n+1}=B_n \cup (t_{n+1}+A_{j_{n+1}}),\quad
M_{n+1}=M_n+3N_{j_{n+1}}.
\end{gather*}
We can find that $B_n\subset B_{n+1}$, $M_n<M_{n+1}$ and $B_{n+1}\subset[1, M_{n+1}]^d$. We iterate this discussion inductively and get a sequence of sets $B_1\subset B_2\subset \cdots$. Define $B=\cup_{n\in \mathbb{N}} B_n$. It follows that
\[
|B\cap [1,M_n]^d|=|B_n \cap [1,M_n]^d |\geq |A_{j_{n}}| >\delta N_{j_{n}}^d\geq \frac{\delta}{4^d} M_{n}^d,
\]
which means that $B$ satisfies (\ref{ap-f6}). Let us show that $B$ does not contain any $(k,0,\mathbf{e})$-APs. Assume that $B$ contains some $(k,0,\mathbf{e})$-AP, then let $P$ be such a $(k,0,\mathbf{e})$-AP. There exists an integer $r$ such that $P\subseteq B_r$. By the choice of $A_{N_r}$, $P\subseteq t_r+ A_{N_r} $ does not hold. If $P$ intersects $B_{r-1}$ and $t_r+A_{N_r}$, then there exist two elements $p,p'\in P$ such that
\[
p\in t_r+A_r,\quad p'\in B_{r-1},\quad p'=p+\Delta e_1
\]
for some $\Delta>0$. Then we have
\[
\Delta \geq t_r=M_{r-1}+2N_{j_r}.
\]
Thus other terms of $P$ do not belong to $B_r$. This is a contradiction. Hence $P\subseteq B_{r-1}$. By iterating this discussion, we conclude that $P\subseteq A_1$, which is a contradiction. Therefore $B$ does not contain any $(k,0,\mathbf{e})$-APs.
\end{proof}

\section*{Acknowledgement}
The author would like to  thank Professor Kohji Mastumoto and Professor Masato Mimura for useful comments. This work is supported by Grantin-Aid for JSPS Research Fellow (Grant Number: 19J20878).

\end{document}